\documentclass[preprint]{imsart}

\RequirePackage[OT1]{fontenc}
\RequirePackage{amsthm,amsmath}
\RequirePackage[numbers]{natbib}
\RequirePackage[colorlinks,citecolor=blue,urlcolor=blue]{hyperref}
\usepackage[boxed]{algorithm2e}
\usepackage{graphicx}
\usepackage{amsmath}
\usepackage{amsthm}
\usepackage{amsfonts}
\usepackage{framed}
\usepackage{mathrsfs}

\theoremstyle{definition}

\newcommand{\abs}[1]{\left\vert#1\right\vert}

\newcommand{\Bin}{\textrm{Bin}}

\newcommand{\adj}{\mathrm{adj}}

\newcommand{\Pbb}{\mathbb{P}}

\newcommand{\Fcal}{\mathcal{F}}

\newcommand{\Rcal}{\mathcal{R}}

\newcommand{\Wcal}{\mathcal{W}}

\newcommand{\Erdos}{Erd\H{o}s}
\newcommand{\Renyi}{R\'{e}nyi~}

\arxiv{arXiv:0000.0000}

\startlocaldefs
\numberwithin{equation}{section}
\theoremstyle{plain}
\newtheorem{thm}{Theorem}[section]
\newtheorem{cor}[thm]{Corollary}
\endlocaldefs

\begin{document}

\begin{frontmatter}
\title{Percolation Threshold Results on \Erdos-\Renyi Graphs:
  an Empirical Process Approach}
\runtitle{Percolation Threshold Results on Erd\H{o}s-R\'{e}nyi Graphs}

\begin{aug}
\author{\fnms{Michael J.} \snm{Kane}\ead[label=e1]{michael.kane@yale.edu}}

\runauthor{M.~J.~Kane}

\affiliation{Yale University}

%
\end{aug}

\begin{abstract}
\end{abstract}

%

\end{frontmatter}

\section{Introduction}

Random graphs and discrete random processes provide a general approach 
to discovering properties and characteristics of random graphs and 
randomized algorithms.  The approach generally works 
by defining an algorithm on a random graph or a randomized algorithm.
Then, expected changes for each step of the process are used to propose
a limiting differential equation and a large deviation theorem 
is used to show that the
process and the differential equation are close in some sense.  In this
way a connection is established between the resulting process's 
stochastic behavior and the dynamics of a deterministic, asymptotic
approximation using a differential equation.  This approach is 
generally referred to as {\em stochastic approximation} and 
provides a powerful tool for understanding the
asymptotic behavior of a large class of processes defined
on random graphs.  However, little work has been done 
in the area of random graph research to investigate the weak limit behavior
of these processes before the asymptotic behavior overwhelms the random
component of the process.  This context is particularly relevant 
to researchers studying news propagation in social 
networks, sensor networks, and epidemiological outbreaks.  In each of these 
applications,
investigators may deal graphs containing tens to hundreds of vertices and
be interested not only in expected behavior over time but also error estimates.


This paper investigates the connectivity of graphs, with emphasis
on \Erdos-\Renyi graphs, near the 
percolation threshold when the number of vertices is not asymptotically large.
More precisely, we define a simple algorithm for simulating directed 
percolations on a graph in Section \ref{sect:dp_algo}. 
Section \ref{sect:overview} provides an overview of the two fundamental 
techniques required for our investigation: stochastic approximation and the 
functional martingale central limit theorem.
In Section \ref{sect:sawlt}, these tools are applied to the directed 
percolation algorithm to show that the behavior of the process converges 
to an ordinary differential equation plus a stretched-out brownian motion.
This result allows us to re-examine many of the classical random graph 
results \cite{Erdos1960} involving the evolution of random graphs near the 
percolation threshold. Furthermore, because the process can be modeled
as a function of a stretched out brownian-motion we can draw on the
stochastic calculus literature to derive new results for random graphs.
For example, in Section \ref{sect:giantStoppingTime} this new 
representation is used to find the percolation threshold for the graph
by deriving the distribution of the stopping time for the algorithm
to percolate over the largest component. 



\section{The Directed Percolation Algorithm} \label{sect:dp_algo}

The percolation algorithm investigated in this paper is defined in
Algorithm \ref{algo:percolation}. The algorithm works on 
a graph with all vertices labelled ``not visited''.  
At time zero one vertex is one labelled ``visited not
transmitted''. The algorithm proceeds by selecting one vertex labelled
``visited not transmitted''. This vertex is labeled ``visited''. The
vertexes neighbors that are labelled ``not visited'' are relabelled
``visited no transmitted''.
If the algorithm progresses to a point where all vertices
are marked either``not visited'' or ``visited transmitted'', then
the algorithm is reseeded by selecting a binomial number of vertices
labelled ``not visited'' and ``visited not transmitted''. The algorithm
continues until all vertices are marked ``visited transmitted''.

\begin{algorithm}[H] 
let $g$ be graph with $n+1$ vertices\\
label all vertices in $g$ ``not visited'' \\
pick one vertex uniformly at random and label it ``visited not transmitted''\\
\While{{\em not all vertices are labelled} ``visited transmitted''}
{
  let $V_n$ be the set of vertices labelled ``not visited'' \\
  let $V_v$ be the set of vertices labelled ``visited not transmitted'' \\
  let $V_t$ be the set of vertices labelled ``visited transmitted'' \\
  \If{$\abs{V_v} > 0$}
  {
    pick a vertex $v$ uniformly at random from $V_v$\\
    label $v$ ``visited transmitted'' \\
    label $\adj(v) \cap V_n$ ``visited not transmitted''
  }
  \Else
  {
    draw $B_i \sim \Bin(p, \abs{V_n})$\\
    pick $B_i$ vertices uniformly at random and label them 
      ``visited not transmitted''
  }
}
\caption{The directed percolation algorithm}
\label{algo:percolation}
\end{algorithm}

One important characteristic of the algorithm is that edges do not need
to be revealed until the algorithm needs to relabel the ``not visited''
vertices adjacent to the selected ``visited not transmitted'' vertex.
This scenario is referred to as the 
{\em method of deferred decision}\cite{Knuth1990} 
and for \Erdos-\Renyi graphs, it induces a binomial 
conditional distribution on the number of vertices whose label 
changes from ``not visited'' to ``visited not transmitted'' at each
step of the algorithm.

\begin{thm}
When the percolation algorithm described in Algorithm \ref{algo:percolation}
is run on an \Erdos-\Renyi graph with $n+1$ vertices
then at iteration $k$ with $0\leq k$ the number of vertices going from
the ``not visited'' to the ''visited not transmitted'' labelling in step 
$k+1$ is distributed as $\Bin(\abs{V_n}, p)$ where $p$ is the 
probability any two vertices are connected.
\end{thm}

\begin{proof}
Let $V_{n,k}$ be the set of vertices labelled ``not visited'' at 
iteration $k$.  If 
there is at least one vertex marked ``visited not transmitted'' then one
of those vertices will be selected for transmission. The edges between
that vertex and its adjacent ``not visited'' vertices are unknown. However,
the probability that it is connected to any one of the ``not visited''
vertices is $p$ and therefore the number of ``not visited'' vertices it 
is connected to, which is the same as the number of new vertices that will
be labelled ``visited not transmitted'' in the next step of the algorithm,
is distributed $\Bin(\abs{V_{n,k}}, p)$. If, on the other hand, there are no
vertices marked ``visited not transmitted'' then, by definition of the 
algorithm, a $\Bin(\abs{V_{n,k}}, p)$ number of vertices labelled 
``not visited'' will be labelled ``visited not transmitted'' in the next step.
\end{proof}

The proof shows that at any step $k$ the number of new vertices that will
be labelled ``visited not transmitted'' at $k+1$ is a binomial 
number depending only on the current number of ``not visited'' vertices and
the connection probability. The aggregate number of vertices labelled 
``visited not transmitted'' and ``visited transmitted'' is strictly 
increasing based on this distribution.

The percolation algorithm on an \Erdos-\Renyi graph can be recast as 
an urn process with one urn holding balls corresponding to vertices labelled
``not visited'' and another holding balls corresponding to vertices labelled
either ``visited not transmitted'' or ``visited transmitted''. Initially,
all $n$ balls are contained in the ``not visited'' urn. Let
$\abs{V_{n,k}}$ be the number of balls in the ``not visited'' urn at time
$k$ with $\abs{V_{n,0}} = 0$ then at each step $\Bin(\abs{V_{n,k}}, p)$
balls are drawn from the ``not visited'' urn and placed in the ``visited''
urn. This urn process is stochastically equivalent to the percolation process.
A formal definition for the urn algorithm is given in Algorithm 
\ref{algo:urn_percolation}.

\begin{algorithm}[H] \label{algo:urn_percolation}
consider two urns labelled ``not visited'' and ``visited'' \\
place $n$ ball into the ``not visited'' urn \\
\While{there are balls in the ``not visited'' urn} 
{
  let $U_n$ be the number of balls in the ``not visited'' urn \\
  draw $b \sim \Bin(U_n, p)$ \\
  move $b$ balls from the ''not visited'' urn to the ``visited'' urn \\
}
\caption{The urn model equivalent of the directed percolation algorithm}
\end{algorithm}

The urn model process provides
a conceptually simpler tool for investigating the behavior of the 
directed percolation process.
It also provides a means for investigating the behavior of the algorithm
near the percolation threshold through the following theorem.

\begin{thm} \label{thm:subcritical}
Consider the urn model process. The event were, at time $k$, the number of 
balls in the ``visited'' urn is less than than $k$, is equivalent to exhausting
the component where the algorithm started. That is, all vertices in the
component are labelled ``visited.''
\end{thm}

\begin{proof}

Consider the directed percolation process on a graph with size greater at least
two. At step zero one a ``seed'' vertex is selected.
At the beginning of step one the seed vertex is chosen for transmission.
If it has no neighbors, then the first component, which consisted of the
seed vertex only is exhausted. Otherwise, without loss of generality,
assume that there is one adjacent vertex, labelled $v_1$. The seed vertex
is no longer considered and $v_1$ is labelled ``visited not transmitted,''
which is equivalent to moving one ball into the ``visited'' urn. Once
again if $v_1$ has no neighbors then, at time step two, the number of 
transmitted vertices is 1 since the seed vertex is not included and no
new vertices are visited. In this case, $k=1$ at time step two corresponding
to the component consisting of the seed vertex and $v_1$ being visited.
This process continues with newly visited vertices corresponding to
moving balls to the ``visited'' urn. The process stops when the graph component
is exhausted, which occurs when the total number of visited vertices is 
less than the time step.
\end{proof}

\begin{figure}
\centerline{
\includegraphics[width=4.5in]{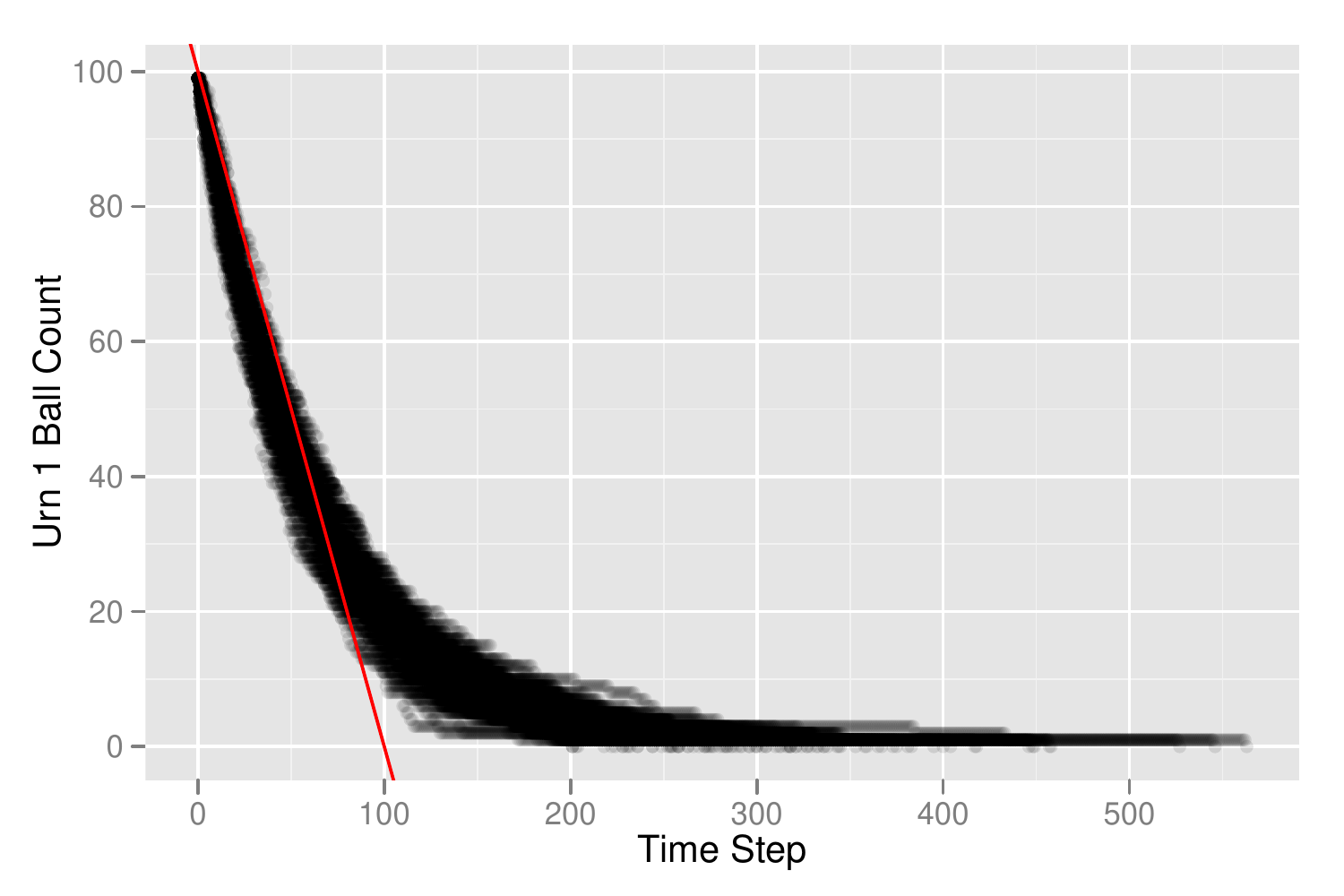}
}
\caption{Visualizing the empirical distribution of the number of balls in
urn 1 over 100 runs near the percolation threshold}
\label{fig:urn1Process}
\end{figure}


Figure \ref{fig:urn1Process} 
shows the number of balls in urn 1 when 
the urn process described in Algorithm \ref{algo:urn_percolation} was simulated
for 100 runs with $p=1.6$ and $n=100$.  A diagonal line is also shown and any
points to the right of the diagonal line correspond to simulated process
whose corresponding percolation process failed to spread to all vertices
in the graph. For this simulation seven of the 100 runs failed.

The figure provides two important insights in understanding the behavior
of the process. First, the process slope is steep compared to the diagonal
at the beginning of the process.  For the urn process, the number of balls 
in urn 1 is large resulting in a large number of balls moving to urn 2
at each time step.  As the process continues there are fewer
balls in urn 1 and as a result fewer balls are moved to urn 2 at each time
step and the slope decreases.
For the graph process, this corresponds to there being
a large number of neighbors for each of the ``visited not transmitted''
vertices. Second, the process variance increases while the slope is steep
compared to the diagonal line and concentrates as the process levels off.
By definition, the number of balls in urn 2 cannot be bigger than $n$.
Each of the processes approach $n$ quickly early in the process and then
more slowly as the process nears the stopped state.

Further simulation results show that as $n$ increases the relative 
variation in the process decreases. That is, the process concentrates
on its expected value at each time step. These expected values over time
can be approximated using a differential equation. The next section provides
the techniques for understanding this concentration phenomena as well as
for finding the corresponding differential equation.

\section{Overview of the Method} \label{sect:overview}

\subsection{Stochastic Approximation}

\subsubsection{Background}

As Wormald \cite{Wormald1999} points out, ``This idea of approximation
has existed in connection with continuous processes... essentially since
the invention of differential equations by Newton for approximation of the
motion of bodies in mechanics.'' However, the term stochastic approximation
was originally coined by \cite{RobinsMonro}. Their paper presented a method
for finding the root of a monotone function under noisy conditions.
Kurtz \cite{Kurtz1970} developed this idea further. 
By imposing appropriate bounds on the difference between steps of the
process he showed that a reparameterized version of the process converges in
distribution to a differential equation. This area of research has remained
active with Darling and Norris \cite{DarlingNorris2008} publishing 
a recent paper providing new conditions for the convergence of stochastic
processes to an ODE based on Gr\"{o}nwall's inequality~\cite{Gronwall1919}
along examples of random processes that lend themselves to thes results.

Stochastic approximation techniques have also been applied to random graph.
Wormald \cite{Wormald1995} uses techniques similar to those provided by
Kurtz to show that the degree of the vertices in a random graph, where
the maximum degree is bounded, converge as the total number of 
vertices gets large. Since then Wormald provided new results 
\cite{Wormald1999} handling the case where processes are not strictly
Markovian. The paper also provided a survey of some graph processes
that are amenable to the stochastic approximation approach including the 
random greedy matching algorithm presented in \cite{KarpSipser1981} and
the degree bounded graph process which Rucinski and Wormald 
\cite{RucinskiWormald1992} used to answer a question originally posed
by \Erdos \  concerning the asymptotic behavior of a degree bounded graph.
Readers interested in an further applications of 
stochastic approximation to examine processes on random graphs
are encouraged to consult Wormald or the recent survey by 
Pemantle \cite{Pemantle2007}.

%
\subsubsection{Overview}

Consider a process $\mathbf{X} = \{X_k : k \in K\}$ where $K$ is the index
set $\{0, 1, ..., n\}$.  Assume that the behavior of $X_k$ is determined by
\begin{equation} \label{eqn:saform}
X_{k+1} = X_k + \varrho_{k+1} 
\end{equation}
where $\varrho_{k+1}$ is a random variable adapted to the natural filtration
of $\mathbf{X}$ up to time $k+1$, which will be denoted $\Fcal_{k+1}$.

The urn process takes values in the integers from zero to $n$
and is defined over all non-negative integers. To derive asymptotic results
it the process is reparameterized, normalizing by $n$
over both the domain and the range. Furthermore, this reparameterized 
process is defined to be cadlag so that its domain and range take values
in the reals. Let $\alpha = k/n$, then the new process $\{Y\}$ can be defined
by
\begin{equation*}
Y_0 = \frac{1}{n} X_0 
\end{equation*}
\begin{equation*}
Y_{\lfloor \alpha n + 1/n \rfloor} = \frac{1}{n} X_k.
\end{equation*}
The reparameterized version process defined in Equation \ref{eqn:saform}
can then be written as
\begin{equation*}
Y_{\lfloor \alpha n + 1 \rfloor} = Y_{\lfloor \alpha n\rfloor} + \frac{1}{n}
  \varrho_{\lfloor \alpha n + 1 \rfloor}
\end{equation*}
or, for notational simplicity
\begin{equation*}
Y_n (\alpha + 1/n) = Y_n(\alpha) + \lambda_n( \alpha + 1/n ).
\end{equation*}
If $\Pbb_{\alpha} \lambda_n( \alpha + 1/n ) = f(Y_n)$ then the reparameterized
process can be further re-written as
\begin{align*}
Y_n (\alpha + 1/n) = Y_n(\alpha) + f(Y_n(\alpha)) - \xi_n(\alpha + 1/n)
\end{align*}
where $\xi_n(\alpha + 1/n)$ is a centered, martingale increment. Now, let
$\{y\}$ be a deterministic analog of the $\{Y\}$ process with $y_0 = Y_0$
and
\begin{equation} \label{eqn:deterministic_analogue}
y_n(\alpha + 1/n) = y_n (\alpha) + f(y_n(\alpha)).
\end{equation}
The difference between $\{Y\}$ and $\{y\}$ at any value of $\alpha$ over 
the domain can be written as
\begin{align*}
\Delta_n(\alpha + 1/n) &= Y_n(\alpha + 1/n) - y_n(\alpha + 1/n) \\
  &= \Delta_n(\alpha) + f(Y_n(\alpha)) - f(y_n(\alpha)) + \xi_n(\alpha + 1).
\end{align*}
When the difference $f(Y_n(\alpha)) - f(y_n(\alpha))$ is small, the 
difference between the process and the deterministic analogue is the
sum of the martingale increments.
\begin{align}
\Delta_n(\alpha + 1/n) &\simeq \Delta_n(\alpha) + \xi_n(\alpha + 1/n) \notag \\ 
  &= \sum_{k=1}^{\alpha n} \xi_n(k/n). \label{eqn:mg_increments}
\end{align}
If the sum of the martingale increments converges to zero and
$f(y_n(\alpha))$
can be approximated arbitrarily well by a differential equation, then
the reparameterized process converges to the differential equation
asymptotically.

\subsection{Functional Martingal Central Limit Theorem}

\subsubsection{Background}

\subsubsection{Overview} \label{sect:fcclt}

In the classic stochastic approximation literature it is assumed that
the term in Equation \ref{eqn:mg_increments} is asymptotically zero. However,
\cite{Hall1980} show that certain martingale increment processes, such
as this one, which are defined over cadlag sample paths,
converges to a stretched-out brownian motion. That is, a brownian motion
$B(t)$ with strictly increasing transformation $H(\cdot)$ to the time scale:
$B(H(t))$. Sufficient conditions for convergence to a stretched-out 
Brownian motion from \cite{Pollard1984} are given here for reference.

\begin{thm} \label{thm:PollardSBM}
Let $\{Z_n(t) : 0 \leq t < \infty\}$ 
be a sequence of martingales adapted to its natural filtration, 
and $\Pbb Z_n(t)^2 < \infty$. Let $\{Z_n\}$ have conditional variance 
process $\{V_n\}$. Let $H$ be a continuous, increasing function on $[0, \infty)$
with $H(0) = 0$. Let $J_T(x)$ be the maximum jump in a sample path
\begin{equation}
J_T(x) = max\{ \abs{ x(s) - x(s-) } : 0 \leq s \leq T\}.
\end{equation}
Sufficient conditions for convergence in distribution of $\{X_n\}$ to a
stretched-out brownian motion $B(H(t))$ are:
\begin{enumerate}
\item $Z_n(0) \rightarrow 0$ in probability
\item $V_n(t) \rightarrow H(t)$ in probability for each fixed $t$
\item $\Pbb (J_k (Z_n) )^2 \rightarrow 0$ for each fixed $k$ as 
  $n \rightarrow \infty$
\end{enumerate}
\end{thm}

\section{Applying Stochastic Approximation and the Functional Central
Limit Theorem to the Urn Process} \label{sect:sawlt}

Returning to the urn process, let $n$ be the total number of balls in urn 1
and urn 2 at any time,
let $p$ the probability any single ball in urn 1 is moved to urn 2, and let
let $U_n(k)$ be the number of balls in urn
1 at time $k$. Then 
\begin{equation} \label{eqn:urn1_process}
U_n(k+1) = U_n(k) - b(k+1)
\end{equation}
where $b(k+1) \sim \Bin (U_n(k), p)$.
Let $S_n(\alpha)$ be the reparameterized process with $S_n(0) = 0$ and
\begin{equation} \label{eqn:urn1_reparam}
S_n(\alpha + 1/n) = S_n(\alpha) - \zeta_n(\alpha + 1/n)
\end{equation}
where $\zeta_n(\alpha + 1/n)$ is an $\Fcal_\alpha$-measurable random variable
with distribution $\Bin(U_n(\alpha n), p) / n - pU_n(\alpha)$. 
Equation \ref{eqn:urn1_reparam} can then be written as
\begin{align}
S_n(\alpha + 1/n) &= S_n(\alpha) - \zeta_n(\alpha + 1/n) \notag \\
  &= S_n(\alpha) - \zeta_n(\alpha + 1/n) - 
    \Pbb_\alpha \zeta_n(\alpha + 1/n) + \Pbb_\alpha \zeta_n(\alpha + 1/n) 
    \notag\\
  &= S_n(\alpha) - p U_n(\alpha n)/n  - \left(
    \zeta_n(\alpha + 1/n) - p U_n(\alpha n)/n \right) \notag \\
  &= q S_n(\alpha) - \delta_n(\alpha + 1/n) \label{align:urn1_recursion}
\end{align}
where $q=1-p$ and $\delta_n(\alpha + 1/n)$ is a martingale increment.

\subsection{Approximating the Process with a Differential Equation} 
\label{sect:sa}

\begin{thm} \label{thm:stoch_approx}
If the sum of the martingale increments up to time $\alpha$ can be bound by 
$o_p(n^{-1/2})$ or
less then the process in Equation \ref{align:urn1_recursion} can be written
as
\begin{equation*}
S_n(\alpha) = e^{-c \alpha} + O_p(n^{-1})
\end{equation*}
\end{thm}
\begin{proof}
By definition $S_n(0) =1$. Therefore $S_n(1/n)$ can be written as:
\begin{equation*}
S_n(1/n) = q S_n(0) + \delta_n(1/n)
\end{equation*}
Likewise
\begin{equation*}
S_n(2/n) = q^2 S_n(0) + \delta_n(2/n) + \delta_n(1/n)
\end{equation*}
From this the process can be written as
\begin{equation} \label{eqn:DoobDecomp}
S_n(\alpha) = q^{\alpha n} + \sum_{i=1}^{\alpha n} \delta_n\left(i/n\right)
  q^{\alpha n - i}.
\end{equation}
The summation in \ref{eqn:DoobDecomp} is a martingale. The absolute value
of this summation is bound by the sum of the absolute values of each of the
summands, which is a submartingale process. Therefore, the supremum of the 
martingale can be bound by
\begin{align}
\Pbb \sup_{t \leq \alpha} \abs{\sum_{i=1}^{\alpha n} 
    \delta\left(i/n\right) q^{\alpha n - i} } &
      \leq \Pbb \sup_{t \leq \alpha} \sum_{i=1}^{\alpha n} 
      \abs{ \delta\left(i/n\right) q^{\alpha n - i} } \notag \\
  &\leq 4 \Pbb \sum_{i=1}^{\alpha n} \frac{npq}{n^2} \label{eqn:Doob}\\
  &\leq \frac{4cq}{n} \notag
\end{align}
where \ref{eqn:Doob} follows by the Doob $L^p$ inequality.
The expected maximum of the martingale increments is converging to zero at
a rate of $1/n$.

The difference between $q^k$ and $e^{-ck/n}$ is of order $O(n^{-2})$.
\begin{align*}
q^k - e^{-kc/n} &= \left(1 - \frac{c}{n}\right)^k - \left(1 - \frac{c}{n}
  \right)^k + O(n^{-2}) \\
  &= O(n^{-2})
\end{align*}
To extend this to the reals it is sufficient to show that for any increment
in the process, the difference between $S_n(\alpha)$ and $e^{-c\alpha}$
is $O(n^{-1})$.
\begin{align*}
e^{-\alpha n}  - S_n(\alpha) &= e^{-\alpha c} - 
  q^{\lfloor \alpha c \rfloor} \\
  &\leq e^{-(\alpha + 1/n) c} - q^{\lfloor \alpha c \rfloor} \\
  &\leq q^{\lfloor \alpha n \rfloor +1} - q^{\lfloor \alpha n\rfloor} \\
  &\leq \frac{c}{n} e^{-\alpha n} \\
  &= O(n^{-1})
\end{align*}
\end{proof}

\subsection{Applying the Functional Martingale Central Limit Theorem}

According to Equation \ref{eqn:urn1_reparam} each increment of the urn
process is a function of the last state of the process minus a martingale
increment. Theorem \ref{thm:stoch_approx} shows that if these martingale
increments are not too big then, in the limit, a reparameterized version
of the process will converge to a differential equation. In this section
it is shown that the martingale process converges to a stretched-out brownian
motion whose variance is decreasing in $n$.

Consider the process in Equation \ref{eqn:urn1_process}. The next urn
count is equal to the current urn count minus a binomial number of balls.
The binomial number is determined by the current number of balls in urn 1
and the probability that a ball is moved from urn 1 to urn 2. This time,
let $\nu_n(k) = b_n(k+1) - \Pbb b_n(k+1)$ and decompose the process in the 
following way
\begin{align*}
\frac{U_n(k+1)}{q^{k+1}} &= \frac{U_n(k) + b_n(k+1)}{q^{k+1}} \\
  &= \frac{U_n(k) - b_n(k+1) - \Pbb b_n(k+1) + \Pbb b_n(k+1)}{q^{k+1}} \\
  &= \frac{U_n(k) - p U_n(k) - \nu_n(k)}{q^{k+1}} \\
  &= \frac{U_n(k)}{q^k} - \frac{\nu_n(k)}{q^{k+1}}.
\end{align*}
Call define a new process $\{T\}$ where
\begin{equation} \label{eqn:T_process}
T_n(k) = \frac{U_n(k)}{\sqrt{n} q^{k}} - \sqrt{n}
\end{equation}
for integers $k \geq 0$. This process is a martingale with strictly 
increasing variance. 

\begin{thm}
The process $\{T\}$ defined in Equation \ref{eqn:T_process} converges to
a stretched-out brownian motion with variance $e^{ck/n} - 1$ at time 
$0 \leq k \leq n$.
\end{thm}
\begin{proof}
Condition 1 of Theorem \ref{thm:PollardSBM} is satisfied by definition of the
process $\{T\}$. Condition 2 can be derived using a conditioning argument.
\begin{align*}
VAR \left(T_n(k) \right) &= \Pbb \sum_{i=1}^k \left( 
    \frac{\nu_n(i)}{\sqrt{n} q^i} \right)^2\\
  &= \frac{1}{n} \Pbb \sum_{i=1}^k \Pbb_{i-1} 
    \left( \frac{\nu_n(i)}{\sqrt{n} q^i} \right)^2 \\
  &= \frac{1}{n} \Pbb \sum_{i=1}^k \frac{pq U_n(i)}{q^{2i}} \\
  &= \frac{1}{n} \sum_{i=1}^k \frac{npq^{i+1}}{q^{2i}} \\
  &= pq \sum_{i=1}^{k-1} q^i + O(n^{-1}) \\
  &= pq \left( \frac{1 - q^{-(k+1)}}{q - 1} \right) + O(n^{-1}) \\
  &= q^{-k} - 1 + O(n^{-1}) \\
  &= e^{ck/n} - 1 + O(n^{-1})
\end{align*}
And condition 3 can be derived by realizing that the largest jump is bound
by the sum of all jumps in the process up to time $k$. Let $\varepsilon > 0$, 
then
\begin{align*}
\Pbb \left\{ \left(\max_{i \leq k} \frac{\nu_n(i)}{\sqrt{n}} \right)^2
  \geq \sqrt{\varepsilon} \right\} &\leq 
    \Pbb \left\{ \left(\max_{i \leq k} \frac{\nu_n(i)}{\sqrt{n}} \right)^4
    \geq \varepsilon \right\} \\
  &\leq \Pbb \left\{ \frac{1}{n} \sum_{i=1}^k \nu_n^2(i) \geq \varepsilon\right\} \\
  &\leq \frac{\Pbb \left( \sum_{i=1}^k \nu_n(i) \right)^4}{n^2 \varepsilon} \\
  &\leq \frac{ \left( \Pbb \sum_{i=1}^k \nu_n^2(i) \right)^2 }{n^2 \varepsilon} \\
  &\leq \frac{n (npq)^2 }{n^2 \varepsilon} \\
  &\leq \frac{(cq)^2}{n}
\end{align*}
which approaches zero as $n\rightarrow \infty$. 
\end{proof}
\begin{cor}
If $\delta_n(\alpha)$ is the martingale increment from Equation
\ref{align:urn1_recursion} then
\begin{equation}
\sum_{\alpha n} \delta_n(i) = n^{-1/2} B\left(e^{-\alpha c} \left(
  1 - e^{-\alpha c} \right) \right) + o_p(n^{-1/2})
\end{equation}
where $B(H(t))$ is a stretched-out brownian motion with variance process $H(t)$.
\end{cor}
\begin{proof}
Recall that $T_n(k)$ converges to a $B(e^{ck/n}-1)$. The martingale 
process $U_n(k)/\sqrt{n}$ also converges to the stretched-out 
brownian motion. 
\begin{align*}
VAR\left(\frac{U_n(k)}{\sqrt{n}} \right) &= q^{2k} \left(e^{ck/n}-1\right) \\
  &= e^{-ck/n} \left(1 - e^{-ck/n} \right)+ O(n^{-1})
\end{align*}
The result follows by realizing that the variance process $\Wcal_n(\alpha)$ is 
half an order of magnitude smaller than $U_n(k)$ and as a result, so is its
standard deviation.
\end{proof}



\section{A General Boundary-Crossing Result for the Directed Percolation
Algorithm}
\label{sect:giantStoppingTime}

\begin{thm} \label{thm:HittingTime}
Let $\Wcal_n(\alpha)$ be the approximation of the process of interest:
\begin{equation} \label{eqn:AproxFunc}
\Wcal_n(\alpha) = v(\alpha) + n^{-1/2} B\left(v(\alpha) (1-v(\alpha))\right)
\end{equation}
where $v(\alpha) = e^{-\alpha c}$ and $B(\alpha)$ is a stretched-out
brownian motion with variance parameter $\alpha$.  Let 
$\alpha_0 = \alpha - \epsilon n^{-1/2}$ with $\epsilon > 0$. 
Let $\tau_A$ be the first time $\Wcal_n(\alpha) > A$ 
for $0 < A \leq 1$. The density of $\tau_A$ is
\begin{equation} \label{eqn:hittingTime}
\tau_A(t) = c \sqrt{\frac{n A}{1-A}}
  \phi\left( \sqrt{\frac{n A}{1-A}} (tc-\log(A)) \right).
\end{equation}
\end{thm}
\begin{proof}
Equation \ref{eqn:AproxFunc} can be expressed as
\begin{equation} \label{eqn:TaylorAproxFunc}
\Wcal_n(\alpha) = v(\alpha_0) + (\alpha - \alpha_0) v'(\alpha_0) +
  \Rcal_d(\alpha) + n^{-1/2} B\left(v(\alpha_0)(1-v(\alpha_0)\right) + 
  \Rcal_r(\alpha)
\end{equation}
where $\Rcal_d(\alpha) = O(n^{-1})$ and can be thought of as the remainder 
of the deterministic portion of $\Wcal_n$ and 
$\Rcal_d(\alpha) = o_p(n^{-1/2})$ is the remainder of the random portion.

First, show that $R_d(\alpha) = O(n^{-1})$ by realizing that the third
order term of the Taylor series expansion of the deterministic part
is
\begin{equation*}
\frac{\left(\alpha - \alpha_0\right)^2}{2} \left(e^{-c \alpha_0}\right)'' =
  \frac{\epsilon}{2n} c^2 e^{-c \alpha_0}
\end{equation*}
since the difference $\alpha - \alpha_0 = \epsilon n^{-1/2}$.

Next show that 
\begin{equation*}
B(v(\alpha) \left(1-v(\alpha)\right)) = B(v(\alpha_0)\left(1-v(\alpha_0) 
  \right) ) + o_p(1)
\end{equation*}
by substituting $\alpha = \alpha_0 + \epsilon n^{-1/2}$
\begin{align*}
B(v(\alpha) \left(1-v(\alpha)\right)) &= B( v(\alpha_0 - \epsilon n^{-1/2})
  (1 - v(\alpha_0 - \epsilon n^{-1/2}) ) \\
  &= B(e^{-c(\alpha_0 + \epsilon n^{-1/2})} - e^{-2c(\alpha_0 + 
    \epsilon n^{-1/2})})
\end{align*}
and showing that the exponential variance terms converge to the original
\begin{align*}
e^{-c(\alpha_0 + \epsilon n^{-1/2})} &= 1 - c(\alpha_0 - \epsilon n^{-1/2}) 
  + \frac{1}{2} c^2 (\alpha_0 - \epsilon n^{-1/2})^2 + ... \\
  &= 1 - c \alpha_0 + \frac{1}{2} (c \alpha_0)^2 + ... + O(n^{-1/2}) \\
  &= e^{-c \alpha_0} + O(n^{-1/2})
\end{align*}

Finally, set $\alpha_0$ to $-\log(A)/c$. 
Substitute into Equation \ref{eqn:TaylorAproxFunc}, note that 
$\Wcal_n(\alpha) = A$ by assumption and $v(\alpha_0) = A$ as a consequence
of the choice for $\alpha_0$. 
\begin{equation}
\alpha = \alpha_0 + \frac{1}{c}
  \sqrt{\frac{(1-v(\alpha_0))}{n v(\alpha_0)}} Z + 
    o_p(n^{-1/2})
\end{equation}
The result shows that $\alpha$ is distributed as normal and is 
centered at $\alpha_0$. The proof follows by realizing that the 
hitting time $\tau_A$ has density equal to that of $\alpha$.
\end{proof}

\section{Applications: Finding the Probability that the Giant Component
has been Exhausted}

The results from Section \ref{sect:giantStoppingTime} can be used to 
get distribution of the time when the giant component is exhausted 
in the directed percolation algorithm defined in Section \ref{sect:dp_algo}.

Equation \ref{eqn:AproxFunc} shows that if and $n$ is large then
the process is approximately equal to it's stochastic approximation, 
$v(\alpha)$. Theorem \ref{thm:subcritical} showed that the first time the
number of ``visited'' vertices is less than the time step corresponds to
exhausting the first component the algorithm percolates over. From these
two results it follows that the first component is exhausted when
\begin{equation*} \label{eqn:crossing}
e^{-c \alpha} = 1 - \alpha
\end{equation*}
when $n$ is big. The result also shows that process is asymptotically 
subcritical when $c \leq 1$.  Since the only solution to
$e^{-c \alpha} = 1 - \alpha$ is $\alpha = 0$. It should be noted that
this result is consistent with the result from the 
original \Erdos-\Renyi paper \cite{Erdos1960} where, asymptotically,
the ratio of the largest component to the total number of vertices in a 
component is $O(\log(n)/n)$ when $c < 1$ and $O(n^{-1/3})$ when $c=1$.

When $n$ is not too large, the results from the previous section give the 
distribution for the first time
the process crosses any pre-determined horizontal line. This result can be
used to find the probability that the giant component has been exhausted
in the directed percolation algorithm \ref{algo:percolation} at any
time step. For a fixed $c$, this is accomplished by numerically solving 
for $\alpha$ in Equation \ref{eqn:crossing}, calculating $A = e^{-c \alpha}$, 
and use Equation \ref{eqn:hittingTime} to get the distribution.

\begin{figure}
\centerline{
\includegraphics[width=4.5in]{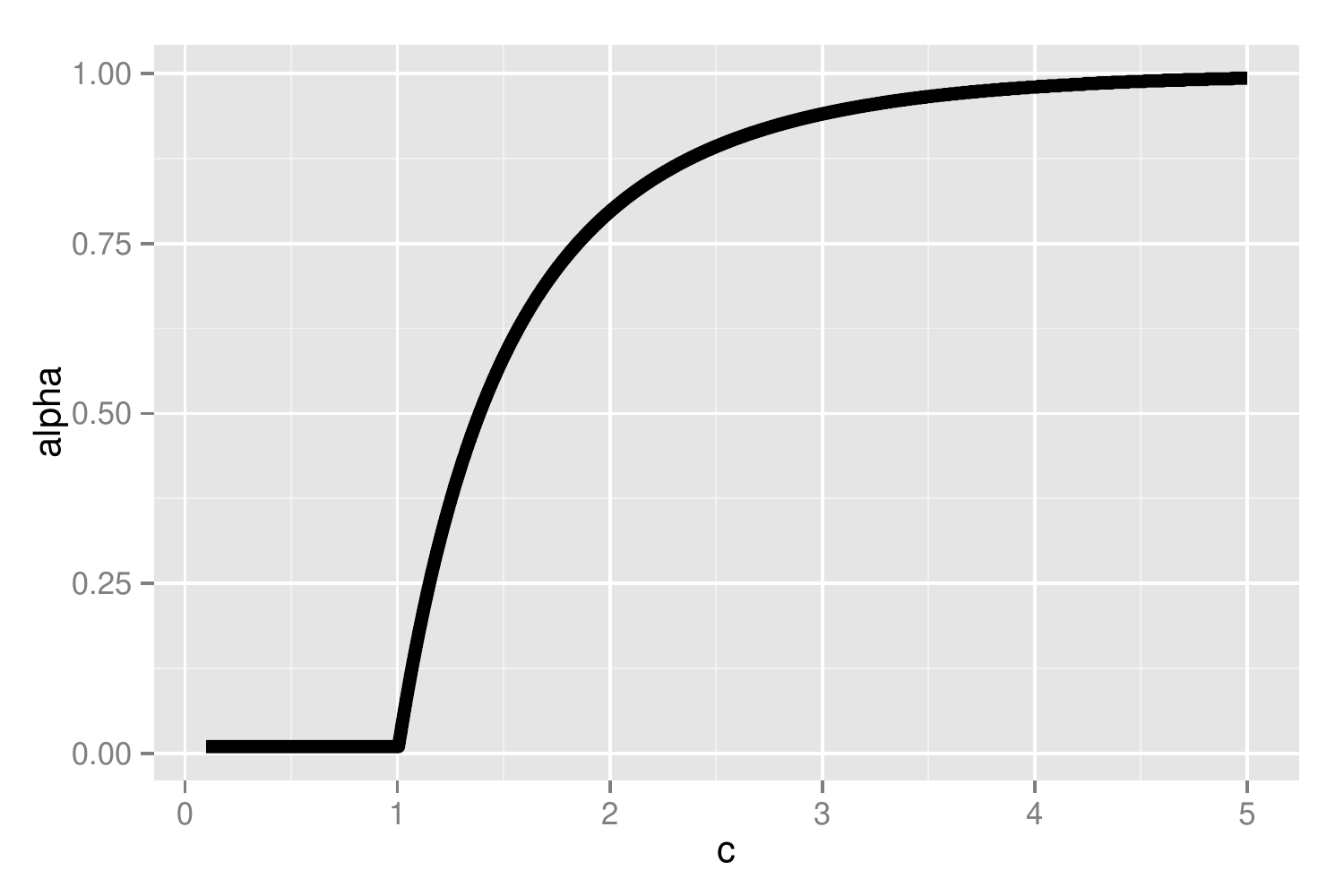}
}
\caption{The solution for $\alpha$ for the function 
$e^{-\alpha c} = 1 - \alpha$.}
\end{figure}

\section{Conclusion}

This paper presents weak limit results for a directed percolation algorithm
on \Erdos-\Renyi graphs. This process concentrates on an ordinary differential
equation and it is shown that, in a pre-asymptotic setting, the process
can be approximated by it's asymptotic ODE plus a stretched-out brownian
motion. While many of the results presented are specific to the choice of
the algorithm and the type of random graph, the underlying approach is more
general. The derived results only require a Lipschitz condition on the
conditional increments of the process along with control over the 
variance of the process. As a result, the techniques used can be seen 
as a general approach to uncovering the characteristics of graphs, 
modeling outbreaks, studying new propagation in social networks, etc.
when the total number of vertices is relatively smaller.

\end{document}